\documentclass[10pt, reqno]{amsart}
\usepackage{amsaddr}
\usepackage{latexsym,amsfonts}
\usepackage{amssymb,amsthm,upref,amscd}
\usepackage[T1]{fontenc}
\usepackage{times}
\usepackage{amscd}
\usepackage{enumerate}
\usepackage{mathrsfs}
\usepackage{amsmath}
\usepackage{color}
\usepackage{soul}
\usepackage{indentfirst}
\usepackage{comment}
\PassOptionsToPackage{reqno}{amsmath}

\newtheorem{thm}{Theorem}[section]

\newtheorem{defi}{Definition}[section]
\newtheorem{lem}{Lemma}[section]

\theoremstyle{theorem}

\theoremstyle{definition}
\newtheorem{rem}{Remark}[section]

\newcommand{\R}{\mathbb{R}}
\newcommand{\C}{\mathbb{C}}
\numberwithin{equation}{section}
\newcommand{\N}{\mathbb{N}}

\newcommand{\eps}{\epsilon}

\newcommand{\wto}{\rightharpoonup}

\makeatletter \@addtoreset{equation}{section} \makeatother

\usepackage[textwidth=138mm,
textheight=215mm,
left=30mm,
right=30mm,
top=25.4mm,
bottom=25.4mm,
headheight=2.17cm,
headsep=4mm,
footskip=12mm,
heightrounded,]{geometry}

\usepackage[colorlinks=true,urlcolor=blue,
citecolor=black,linkcolor=black,linktocpage,pdfpagelabels,
bookmarksnumbered,bookmarksopen]{hyperref}
\newcounter{const}
\author[D. Garrisi]{Daniele Garrisi}

\address[D. Garrisi]{\centerline{Room 324, Sir Peter Mansfield Building}
\centerline{School of Mathematical Sciences}
\centerline{University of Nottingham Ningbo China}
\centerline{199 Taikang East Road}
\centerline{315100, Ningbo, China}}

\author[T. Gou]{Tianxiang Gou}

\address[T. Gou]{%
\centerline{School of Mathematics and Statistics}
\centerline{Xi’an Jiaotong University}
\centerline{710049, Xi’an, Shaanxi, China}}



\email{daniele.garrisi@nottingham.edu.cn}
\email{tianxiang.gou@xjtu.edu.cn}
\title[Orbital stability and concentration]{Orbital stability and concentration of standing waves to nonlinear Schr\"odinger systems with mass critical exponent}
\thanks{
The work has been supported by the \textsl{PDE Research Group} of School of Mathematical Sciences of the University of Nottingham Ningbo China and funded by the \textsl{FoSE New Researchers Grant}. The second author was supported by the National Natural Science Foundation of China (No. 12101483) and the Postdoctoral Science Foundation of China (No. 2021M702620).}

\begin{document}
\begin{abstract}
For a nonlinear Schr\"odinger system with mass critical exponent, we prove the existence and orbital stability of standing-wave solutions obtained as minimizers of the underlying energy functional restricted to a double mass constraint. In addition, we discuss the concentration of a sequence of minimizers as their masses approach to certain critical masses.
\end{abstract}
\maketitle
\section{Introduction}
In this paper, we are concerned with standing waves to the following time-dependent nonlinear Schr\"odinger system in $\R \times \R^N$,
\begin{align} 
\label{sys}
\left\{
\begin{aligned} 
\textnormal{i}\partial_t\phi_1(t,x) + \Delta_x \phi_1(t,x) + \mu_1 |\phi_1|^{\frac 4N}\phi_1 + 
\beta r_1 |\phi_1|^{r_1 -2}\phi_1 |\phi_2|^{r_2} &= 0,\\
\textnormal{i}\partial_t\phi_2(t,x) + \Delta_x \phi_2(t,x) + \mu_2 |\phi_2|^{\frac 4N}\phi_2 + 
\beta r_2 |\phi_2|^{r_2 -2}\phi_2 |\phi_1|^{r_1} &= 0,
\end{aligned}
\right.
\end{align}
where $N \geq 1$, $\mu_1, \mu_2, \beta>0$, $r_1, r_2>1$, $r_1+r_2<2+ \frac 4 N$ and the wave functions $\phi_1$ and $\phi_2$ are defined on $\R \times \R^N$. 
This system has many applications in mean field models for binary mixtures of 
Bose-Einstein condensates or for binary gases of fermion atoms in degenerate quantum states (Bose–Fermi mixtures, Fermi–Fermi mixtures), see \cite{BFK15,BBE97,Mal08}. Here a standing wave to \eqref{sys} is defined by
$$
\phi_1(t,x) = e^{\textnormal{i}\lambda_1 t}u_1(x), \quad \phi_2(t, x)=e^{\textnormal{i}\lambda_2 t}u_2(x), \quad \lambda_1, \lambda_2 \in \R.
$$
This ansatz leads to the following nonlinear elliptic system satisfied by $u_1$ and $u_2$,
\begin{align}  
\label{elliptic-sys}
\left\{
\begin{aligned} 
-\Delta u_1 + \lambda_1 u_1 = \mu_1 |u_1|^{\frac 4N}u_1 + \beta r_1 |u_1|^{r_1 -2} u_1 |u_2|^{r_2},\\
-\Delta u_2 + \lambda_2 u_2 = \mu_2 |u_2|^{\frac 4N}u_2 + \beta r_2 |u_2|^{r_1 -2} u_2 |u_1|^{r_2}.
\end{aligned}
\right.
\end{align}
Note that the masses of solutions to the Cauchy problem for \eqref{sys} are conserved along time, i.e. for any $t \in [0, T)$,
$$
\int_{\R^N} |\phi_1(t, \cdot)|^2 \, dx=\int_{\R^N} |\phi_1(0, \cdot)|^2 \,dx, \quad \int_{\R^N} |\phi_2(t, \cdot)|^2 \, dx=\int_{\R^N} |\phi_2(0, \cdot)|^2 \,dx.
$$
Thus it is interesting to consider solutions to \eqref{elliptic-sys} having prescribed $L^2$-norm constraints. The study in the present paper is devoted to this direction, namely we shall investigate solutions to \eqref{elliptic-sys} under the following $L^2$-norm constraints,
\begin{align}\label{mass}
\int_{\R^N} |u_1|^2 \,dx=a_1>0, \quad \int_{\R^N} |u_2|^2 \,dx=a_2>0.
\end{align}
Solutions to \eqref{elliptic-sys} and \eqref{mass} are often referred to as normalized solutions, which formally correspond to critical points of the energy functional $J$ \begin{align*} 
\begin{split}
J(u_1, u_2) :=&\,\frac 12 \int_{\R^N} \left(|\nabla u_1|^2+|\nabla u_2|^2\right)dx -
\frac{\mu_1 N}{2N+4} \int_{\R^N} |u_1|^{2+\frac 4N} \,dx\\
\quad &-\frac{\mu_2 N}{2N+4} \int_{\R^N} |u_2|^{2 +\frac 4N}\,dx
-\beta \int_{\R^N} |u_1|^{r_1}|u_2|^{r_2} \,dx
\end{split}
\end{align*}
on the constraint
\[
S(a):=\left\{u\in H^1(\R^N;\mathbb{C}) \ | \ \|u_i\|_{2}^2 = a_i\text{ for }i=1,2\right\}.
\]
In this situation, the parameters $\lambda_1$ and $\lambda_2$ arsing as Lagrange multipliers related to the constraint $S(a_1, a_2)$ are unknown. For the study of normalized solutions to nonlinear Schr\"odinger system \eqref{elliptic-sys}, most of papers are concerned with mass subcritical or supercritical exponent, see for instance \cite{BJS16,BJ18,BS17,BS19,BS11a,BS11b,BZZ21,CCW11,CZ21,Gou18,GJ16,LZ21,NW11,NW13,NW16,YZ22} and references therein. As far as we know, solutions to \eqref{elliptic-sys} with mass critical exponent have not been considered in the literature, so far. Throughout the paper, we shall assume that the exponents and parameters satisfy the assumption:
\begin{equation}
\label{AH}
\tag{H}
 N \geq 1\ \mu_1, \mu_2, \beta>0,\ r_1, r_2>1,\ r_1+r_2<2 +\frac 4N.
\end{equation}
In some cases, we will also require the following additional assumptions:
\begin{align}
\label{A1}
\tag{A1}
N &= 1,2,\ r_2 < 2\ \mbox{or}\ 
N \geq 3,\ r_2 < 2 \ \mbox{and}\ \frac{2r_1}{2-r_2}\leq\frac{2N}{N - 2},\\
\label{A2}
\tag{A2}
N &= 1,2,\ r_1 < 2\ \mbox{or}\ 
N \geq 3,\ r_1 < 2 \ \mbox{and}\ \frac{2r_2}{2-r_1}\leq\frac{2N}{N - 2}.
\end{align}
Here $2 + \frac 4N$ is the so-called mass critical exponent in the sense of the following Gagliardo-Nirenberg inequality,
\begin{align} \label{gn}
\int_{\R^N} |u|^p \, dx \leq \frac{p}{2\|Q_p\|_2^{p-2}} \left(\int_{\R^N}|\nabla u|^2 \,dx \right)^{\frac{N(p-2)}{4}} \left(\int_{\R^N}|u|^2 \, dx\right)^{\frac p 2-\frac{N(p-2)}{4}},
\end{align}
where \(u\in H^1(\R^N)\) and $2 \leq p<\infty$ if $N=1,2$ and $2 \leq p \leq 2^*:=\frac{2N}{N-2}$ if $N \geq 3$. The equality in \eqref{gn} holds for $p>2$ if and only if $u=Q_p$, where $Q_p \in H^1(\R^N)$ is the unique positive solution to the equation
\begin{align} \label{equq}
-\frac{N(p-2)}{4} \Delta Q + \left(1-\frac{(N-2)(p-2)}{4}\right)Q=|Q|^{p-2}Q \quad \mbox{in} \,\, \R^N.
\end{align}
For simplicity, we shall denote $Q_{2+4/N}$ by $Q$ in what follows. Define
\begin{gather}
\label{eq.gn-critical}
Q_i:= Q\mu_i ^{-N/4},\quad
a^*_i:=\mu^{-N/2}_i\|Q\|_2^2,\quad i=1,2.
\end{gather}
Direct calculations show that for \(i=1,2\), there holds
\[
\|Q_i\|_2^2 = a_i^*,\quad \frac{1}{2}\int_{\R^N} |\nabla Q_i|^2\,dx = 
\frac{\mu_i N}{2N + 4}\int_{\R^N} |Q_i|^{2 + \frac4N}\,dx.
\]
To explore solutions to \eqref{elliptic-sys}-\eqref{mass} under the assumption $(H)$, we shall introduce the following minimization problem,
\begin{align} \label{min}
m(a_1, a_2):=\inf_{(u_1, u_2) \in S(a_1, a_2)} J(u_1, u_2).
\end{align}
Since both the functional and the constraint functions are in \(C^1 (H^1\times H^1;\R)\), 
any minimizer to \eqref{min} is a weak solution to \eqref{elliptic-sys}. The first result reads as follows:
\begin{thm} \label{concentration-compactness}
Assume \eqref{AH} holds and \(1\leq N\leq 4\). Then any minimizing sequence to \eqref{min} is compact in $H^1(\R^N) \times H^1(\R^N)$ up to translations, provided one of the following
conditions holds:
\begin{enumerate}
\item \(0 < a_i < a_i^*\) for \(i=1,2\), 
\item $0<a_1 < a_1^*$ and $a_2=a_2^*$, under the additional assumption (A1),
\item $a_1=a_1^*$ and $0<a_2<a_2^*$, under the additional assumption (A2).
\end{enumerate}
\end{thm}
\begin{rem}
The additional assumptions \eqref{A1} and \eqref{A2} for the cases when one of the masses is critical are technical, and are applied to derive the boundedness of minimizing sequences to \eqref{min} in $H^1(\R^N) \times H^1(\R^N)$.
\end{rem}
In the spirit of the well-known Lions concentration compactness principle presented in \cite{Lio84a,Lio84b}, to prove Theorem \ref{concentration-compactness}, we first need to exclude vanishing of any minimizing sequence to \eqref{min}. This can be done by using \cite[Lemma I.1]{Lio84b} and the fact that $m(a_1, a_2)<0$, see Lemma~\ref{inf}. Then we need to exclude dichotomy. To do this, it is required to establish the following strict subadditivity inequality:
\begin{align}\label{sub}
m(a_1, a_2) <m(b_1, b_2) +m(c_1, c_2),
\end{align}
where $a_i = b_i +c_i $ and $b_i, c_i\geq 0 $ for \(i=1,2\), and both \((b_1,b_2)\) and \((c_1,c_2)\) are different from \((0,0)\). In the case of
nonlinear Schr\"odinger equations, one can adapt scaling techniques to achieve strict subadditivity inequalities, see \cite{BBGM07,BS11a,BS11b,CJS10,Shi14} and references therein. However, when it comes to the case of nonlinear Schr\"odinger systems, scaling techniques are no longer suitable to check strict subadditivity inequalities. This indeed makes the study of the compactness of minimizing sequence more difficult. In this situation, the authors in \cite{Gou18,GJ16} employed the coupled rearrangement arguments originating from \cite{Shi17} to deal with dichotomy in the mass subcritical case, which avoids to demonstrate the strict subadditivity inequality. 
The argument of \cite{Gou18,GJ16} relies on the fact that the infimum of the scalar minimization problem is negative. Nevertheless, this is not the case in our context, see Lemma \ref{smin}, because the problem under our consideration has 
the mass critical exponent. For this reason, we shall make use of \cite[Lemma 3.1]{Gar12} to discuss dichotomy of any minimizing sequence to \eqref{min}.
In \cite{Gar12}, the problem was actually treated in the mass subcritical case, then the adaption of the arguments to ours is nontrivial and additional effort is needed. Indeed, with the help of \cite[Lemma 3.1]{Gar12}, we are able to show that any minimizing sequence to \eqref{min} is convergent in $L^p(\R^N) \times L^p(\R^N)$ for any $2<p<\infty$ if $N=1,2$ and for any $2<p<2^*$ if $N \geq 3$. To obtain the desired conclusion, it suffices to prove that the sequence is also convergent in $L^2(\R^N) \times L^2(\R^N)$. 
To this purpose, we need to establish the following strict inequality,
\begin{align} \label{ineq}
m(a_1, a_2)<m(b_1, b_2),
\end{align}
where $0<b_i \leq a_i$ for $i=1,2$ and $(a_1, a_2) \neq (b_1, b_2)$. If $b_1<a_1$ and $b_2<a_2$, by Lemma \ref{inf}, then we can easily infer that \eqref{ineq} holds true for any $N \geq 1$. However, if $b_1=a_1$ and $b_2<a_2$ or $a_1<b_1$ and $a_2=b_2$, then the verification of \eqref{ineq} becomes tough, because of the presence of mass critical exponent. In these two cases, we can only prove that the strict inequality holds true for $1 \leq N \leq 4$, and it is unclear whether it remains true for any $N \geq 5$. This is the only reason that we assume that $1 \leq N \leq 4$. 

To introduce our result on orbital stability of the set of minimizers, we set
\[
G(a_1, a_2):=\left\{(u_1, u_2) \in S(a_1, a_2)\mid J(u_1, u_2)=m(a_1, a_2)\right\}
\]
and define the following scalar product in $H^1(\R^N) \times H^1(\R^N)$,
\begin{align} \label{scalar}
(u,v)_{H^1\times H^1} := \sum_{i = 1}^2 \text{Re}\int_{\R^N} u_i(x) \overline{v}_i (x)\,dx + \sum_{i = 1}^2 \text{Re}\int_{\R^N} \nabla u_i(x) \overline{\nabla v}_i (x)\,dx.
\end{align}
The Cauchy problem for \eqref{sys} is locally well-posed in the space 
$H^1(\R^N) \times H^1(\R^N)$, in the sense of \cite[Remark~3.5,~p.~126]{Tao06}:
for every initial datum $(\phi_{1,0}, \phi_{2,0})$ in $ H^1(\R^N) \times H^1(\R^N)$, there exist a constant $T > 0$ and a unique solution $(\phi_1, \phi_2)$ in $ C_t H^1_x ([0,T)\times\R^N) \times C_t H^1_x ([0,T)\times\R^N)$ to \eqref{sys} such that \(\phi_1 (t,0) = \phi_{1,0}\) and \(\phi_2 (t,0) = \phi_{2,0}\).
Let $d(\cdot, \cdot)$ be the metric induced by the scalar product \eqref{scalar}.
\begin{defi}
A subset $G \subset H^1(\R^N)\times H^1(\R^N)$ is orbitally stable if for any $\varepsilon > 0$, there exists a constant $\delta > 0$ such that,
whenever \(d((\phi_{1,0}, \phi_{2,0}), G) < \delta\), there holds \(\sup_{t \geq 0}d((\phi_1(t, \cdot), \phi_2(t, \cdot)), G) < \varepsilon\).
\end{defi}
\begin{thm} \label{stability}
Under the assumptions of Theorem~\ref{concentration-compactness}, $G(a_1, a_2)$ is orbitally stable.
\end{thm}
In Theorem \ref{stability}, the fact that $G(a_1, a_2) \neq \emptyset$ is guaranteed by the compactness of minimizing sequences
to \eqref{min}, obtained in Theorem \ref{concentration-compactness}. Utilizing \cite[Theorem 1.1]{BZ88} and \cite[Theorem 3.4]{LL01}, we can prove additional features of 
minimizers to \eqref{min}.
\begin{thm} 
\label{thm}
Let $(u_1, u_2) \in S(a_1, a_2)$ be a minimizer to \eqref{min}. Then there are \(x_0\) in \(\R^N\) and \(\theta_1, \theta_2\) in \(\R$ such that
$$
(u_1(x), u_2(x)) = (e^{\textnormal{i}\theta_1} R_1 (x - x_0), e^{\textnormal{i}\theta_2} R_2 (x - x_0)), \quad x\in\R^N,
$$
where \(R_1,R_2\in C^2(\R^N)\), are positive and radially symmetric-decreasing.
\end{thm}
Finally, we prove an asymptotic behaviour of minimizers to \eqref{min}, as its masses approach to critical masses in both components.
\begin{thm} \label{concentration-blowup}
Let $\{(u_{1,n},u_{2,n})\} $ be a sequence such that \((u_{1,n},u_{2,n})\in G(a_n)\), where \(\{a_n\}\) is such that 
\(0 < a_{i,n}\leq a_i^*\) for \(i=1,2\), and 
\((a_{1,n},a_{2,n})\neq (a_1^*,a_2^*)\). 
Suppose also that $ (a_{1,n},a_{2,n})\to (a_1^*, a_2^*)$ as $n \to \infty$. 
Then there exist a sequence $\{y_n\} \subset \R^N$, $\theta_1, \theta_2 \in \R$ and $\lambda_1, \lambda_2 \geq 0$ satisfying
\[
a_1^*\lambda_1+a_2^*\lambda_2=\frac2 N
\]
such that
\[
\eps_n^{\frac N 2}u_{i,n}(\eps_n(\cdot - y_n)) \to e^{\textnormal{i}\theta_i} \mu_i^{-\frac N 4} \alpha_i^{\frac N 4} Q\left(\alpha_i^{\frac 12} x\right) \,\, \mbox{in} \,\,H^1(\R^N) \,\, \mbox{as} \,\, n \to \infty,
\]
where $\alpha_i=\frac{\lambda_i N}{2}$ for $i=1,2$, $Q $ is the unique positive solution to \eqref{equq} with $p=2+\frac 4N$ and 
$$
\eps_n:=\left(\int_{\R^N} |\nabla u_{1,n}|^2+|\nabla u_{2,n}|^2\, dx\right)^{-\frac 12}=o_n(1).
$$
\end{thm}
\subsection*{Notations} Throughout the paper, $L^p(\R^N)$ denotes the usual Lebesgue space equipped with the norm
$$
\|u\|_p:=\left(\int_{\R^N}|f(x)|^p \, dx \right)^{\frac{1}{p}}, \quad 1 \leq p<\infty, \quad  \|u\|_\infty:= \underset{x\in \R^N}{\mbox{ess sup}} \, |u(x)|.
$$
Moreover, $H^1(\R^N)$ denotes the usual Sobolev space equipped with the norm
$$
\|u\|_{H^1}^2 :=\|u\|_2^2 + \|\nabla u\|_2^2.
$$
We denote by $B_R(x)$ the ball in $\R^N$ with the center $x$ and radius $R > 0$. We use the notation $o_n(1)$ to denote any quantity which tends to zero as $n$ goes to infinity.
\section{Proofs of main results}
In this section, we shall assume that (H) holds and prove Theorems \ref{concentration-compactness}-\ref{concentration-blowup}. First of all, by \eqref{gn} and H\"older's inequality, one can easily check that $J$ is well-defined in $H^1(\R^N) \times H^1(\R^N)$. In fact, as already observed in \cite{GJ16}, the term \(|u_1|^{r_1} |u_2|^{r_2}\) is
integrable on \(\mathbb{R}^N\) under the assumption (H). Since the term is also \(C^1\), the functional \(J\) is differentiable. For a proof, we refer to
\cite[Theorem~2.6,~p.~17]{AP93}.
\subsection{Concentration-Compactness of minimizing sequences} It is convenient to introduce the following scalar minimization problem,
\begin{equation*}
m_{\mu}(a):=\inf_{u \in S_1 (a)} I_{\mu}(u),
\end{equation*}
where 
\begin{equation}
\label{eq.single}
\begin{split}
I_{\mu}(u) &:=\frac 12 \int_{\R^N} |\nabla u|^2\, dx -\frac{\mu N}{2N+4} \int_{\R^N} |u|^{2+\frac 4 N} \, dx,\\
S_1 (a) &:= \{u\in H^1 (\R^N;\C) \mid \|u\|_2^2 = a\}.
\end{split}
\end{equation}
For any $u \in S_1 (a)$ and $t>0$, we introduce a conformal rescaling of \(u\) as
\[
u_t(x) = t^{\frac N 2} u(tx), \quad x \in \R^N.
\]
From \eqref{gn}, we are able to derive the following result.
\begin{lem} 
\label{smin}
Let $\mu>0$, and set $a^* := \mu^{-\frac N 2}\|Q\|_2^2 $ such that $m_{\mu}(a)=0$ for any $0<a \leq a^*$ and $m_{\mu}(a)=-\infty$ for any $a>a^*$.
\end{lem}
We are going to collect some basic properties with respect to the minimization problem \eqref{min}. 
\begin{lem} \label{inf} 
Assume \eqref{AH} holds.
\begin{enumerate}
\item For  $0<a_1<a_1^*$ and $0<a_2<a_2^*$, then $-\infty<m(a_1, a_2)<0$ and minimizing sequences to \eqref{min} are bounded in $H^1(\R^N) \times H^1(\R^N)$.
\item For $0<a_1<a_1^*$ and $a_2=a_2^*$, under the additional assumption \eqref{A1}, we have 
$-\infty<m(a_1, a_2)<0$ and any minimizing sequence to \eqref{min} is bounded in $H^1(\R^N) \times H^1(\R^N)$. The same conclusion holds
for $a_1=a_1^*$ and $0<a_2<a_2^*$, under the additional assumption \eqref{A2}.
\item In all the other cases, that is $(a_1, a_2)=(a_1^*, a_2^*)$, and the case where $a_i>a_i^*$ for some $1 \leq i\leq 2$, we have $m(a_1, a_2)=-\infty$.
\item Let $\{(a_{1,n}, a_{2,n})\} \subset \R^+ \times \R^+$ be a sequence such that $0<a_{1,n} \leq a_1^*$ and $0<a_{2,n} \leq a_2^*$ with $(a_{1,n}, a_{2,n}) \neq (a_1^*, a_2^*)$ and $(a_{1,n}, a_{2,n}) \to (a_1^*, a_2^*)$ as $n \to \infty$, then
$$
\lim_{n\to\infty} m(a_{1,n}, a_{2,n}) = -\infty.
$$
\item
Let $ 0 < a_1\leq a_1^* $ and $0<a_2 \leq a_2^*$ with 
$(a_1, a_2) \neq (a_1^*, a_2^*)$. If 
$a_1=b_1+c_1$ and $a_2=b_2+c_2$ with $b_1, b_2, c_1, c_2 \geq 0$, then
\[
m(a_1, a_2) \leq m(b_1, b_2) + m(c_1, c_2).
\]
\end{enumerate}
\end{lem}
\begin{proof}
(i). For any $(u_1, u_2) \in S(a_1, a_2)$ and $t>0$, we have
\begin{equation}
\label{eq.conformal}
(u_{1,t}, u_{2,t})=(t^{\frac N 2} u_1(tx), t^{\frac N 2} u_2(tx)) \in S(a_1, a_2)
\end{equation}
and
\begin{align*} 
\begin{split}
J(u_{1, t}, u_{2, t})&=\frac {t^2}{2} \int_{\R^N}\big(|\nabla u_1|^2+|\nabla u_2|^2\big)dx-\frac{t^2 N}{2N+4} \int_{\R^N}
\big(\mu_1 |u_1|^{2+\frac 4 N} + \mu_2 |u_2|^{2 +\frac 4N}\big) \, dx \\
&\quad -t^{\frac N 2 (r_1 + r_2 -2)}\beta \int_{\R^N} |u_1|^{r_1}|u_2|^{r_2} \,dx.
\end{split}
\end{align*}
Note that $0<\frac N 2 (r_1+r_2-2)<2$ from \eqref{AH}. Then $J(u_{1,t}, u_{2,t})<0$ for any $t>0$ small enough. This implies that $m(a_1, a_2)<0$ for any $a_1, a_2>0$. Using \eqref{gn} and H\"older's inequality, we find that
\begin{align*} 
\begin{split}
J(u_1, u_2) &\geq \frac 12 \left(1 -\mu_1 a_1^{\frac 2N}\|Q\|_{2}^{-\frac 4 N} \right) \int_{\R^N}|\nabla u_1|^2 \, dx+\frac 12 \left(1 -\mu_2 a_2^{\frac 2N}\|Q\|_{2}^{-\frac 4 N}\right) 
\int_{\R^N}|\nabla u_2|^2 \, dx \\
& \quad -C_1\left(\int_{\R^N} |\nabla u_1|^2+|\nabla u_2|^2\, dx \right)^{\frac{N(r_1 + r_2 -2)}{4}}
\end{split}
\end{align*}
where $C_1=C_1(N, a_1, a_2, r_1,r_2,\beta)>0$. Since $0<\frac N 4(r_1+r_2-2)<1$, then $J$ restricted on $S(a_1, a_2)$ is coercive for any $0<a_1<a_1^*$ and $0<a_2<a_2^*$. This shows that $m(a_1, a_2)>-\infty$ and any minimizing sequence to \eqref{min} is bounded in $H^1(\R^N) \times H^1(\R^N)$ for any $0<a_1<a_1^*$ and $0<a_2<a_2^*$.

\noindent (ii).
For simplicity, we only show the proof for the case $0<a_1<a_1^*$ and $a_2=a_2^*$. In view of H\"older's inequality and \eqref{gn}, we first conclude that
\begin{equation} 
\label{bdd}
\begin{split}
\int_{\R^N} |u_1|^{r_1}|u_2|^{r_2} \,dx &\leq (a_2^*)^{\frac{r_2}{2}} \left(\int_{\R^N}|u_1|^{\frac{2r_1}{2-r_2}} \, dx\right)^{\frac{2-r_2}{2}} \leq C_2 \left(\int_{\R^N} |\nabla u_1|^2 \, dx \right)^{\frac{N(r_1 + r_2 -2)}{4}},
\end{split}
\end{equation}
where $C_2=C_2(N, a_1, a_2^*, r_1, r_2)>0$. Moreover, by Lemma \ref{smin}, we observe that
\begin{equation*}
\begin{split}
J(u_1, u_2) &\geq \frac 12 \int_{\R^N}|\nabla u_1|^2 \, dx -\frac{\mu_1N}{2N+4} \int_{\R^N} |u_1|^{2 +\frac 4N} \, dx -\beta \int_{\R^N} |u_1|^{r_1}| u_2|^{r_2} \, dx.
\end{split}
\end{equation*}
By applying \eqref{gn} and \eqref{bdd}, we have
\begin{equation}\label{eq.inf.2}
\begin{split}
J(u_1, u_2) &\geq \frac 12 \int_{\R^N}|\nabla u_1|^2 \, dx -\frac{\mu_1N}{2N+4} \int_{\R^N} |u_1|^{2 +\frac 4N} \, dx -C_2 \left(\int_{\R^N} |\nabla u_1|^2 \, dx \right)
^{\frac{N(r_1 + r_2 - 2)}{4}}\\
&\geq \frac 12 \left(1 -\mu_1 a_1^{\frac 2N}\|Q\|_{2}^{-\frac 4 N}\right) \int_{\R^N}|\nabla u_1|^2 \, dx -C_2 \left(\int_{\R^N} |\nabla u_1|^2 \, dx \right)
^{\frac{N(r_1 + r_2 - 2)}{4}}.
\end{split}
\end{equation}
Since $0<\frac N 4 (r_1+r_2-4)<1$, then $m(a_1, a_2^*) >-\infty$ for any $0<a_1<a_1^*$. Therefore, we have that $-\infty<m(a_1, a_2^*) <0$ for any $0<a_1<a_1^*$. Let us now prove the boundedness of any minimizing sequence to \eqref{min}. Suppose that $\{(u_{1,n}, u_{2,n})\} \subset S(a_1, a_2^*)$ is a minimizing sequence to \eqref{min}. From \eqref{eq.inf.2}, we know that $\{u_{1,n}\}$ is bounded in $H^1(\R^N)$. Let us now show that $\{u_{2,n}\} \subset S_1 (a_2^*)$ is bounded in $H^1(\R^N)$ as well. To do this, we shall argue by contradiction 
and suppose that $\{u_{2,n}\}$ is unbounded in $H^1(\R^N)$. Then we may assume that $\|\nabla u_{2, n}\|_2 \to \infty$ as $n \to \infty$.
Define 
$$
v_{2, n}:=(u_{2,n})_{t_n}, \quad t_n:=\frac{1}{\|\nabla u_{2,n}\|_2}=o_n(1).
$$
It is straightforward to check that $v_{2,n} \in S_1 (a_2^*)$ and $\|\nabla v_{2,n}\|_2=1$. 
According to \cite[Proposition~3.1]{Mer93}, there exist \(\{x_n\}\) and \(\{\theta_n\}\) such that 
\(e^{\textnormal{i}\theta_n}v_{2,n}(\cdot - x_n)\to Q_2\) in \(L^2(\mathbb{R}^N)\), as \(n\to\infty\), where 
\(Q_2\) is defined in \eqref{eq.gn-critical}.
Up to extract a subsequence, we can suppose that the convergence is pointwise almost everywhere and dominated by a function \(g\in L^2(\mathbb{R}^N)\).
There exist \(C_3 = C_3 (r_2)\) such that
\begin{equation}
\label{eq.inf.3}
\begin{split}
&\big||v_{2,n}(\cdot - x_n)|^{r_2} - Q_2^{r_2}\big|^{\frac{2}{{r_2}}} \to 0 \text{ a.e.}\\
&\big||v_{2,n}(\cdot - x_n)|^{r_2} - 
Q_2^{r_2}\big|^{\frac{2}{{r_2}}}\leq C_3 (g^{2} + Q_2^{2}).
\end{split}
\end{equation}
Observe that
\begin{equation}
\label{eq.inf.6}
\begin{split}
\int_{\R^N} |u_{1,n}|^{r_1} |u_{2,n}|^{r_2} dx &= \int_{\R^N} |u_{1,n}|^{r_1}|t_n^{-\frac{N}{2}} v_{2,n}(t_n^{-1}\cdot)|^{r_2}  dx \\
&=\int_{\R^N} t_n^{-\frac{Nr_2}{2} + N}|u_{1,n}(t_n\cdot))|^{r_1} |v_{2,n}|^{r_2} dx \\
&=\int_{\R^N} t_n^{-\frac{Nr_2}{2} + N}|u_{1,n}(t_n (\cdot - x_n)|^{r_1} |v_{2,n}(\cdot - x_n)|^{r_2} dx \\
&=\int_{\R^N} |t_n^\alpha u_{1,n}(t_n (\cdot - x_n))|^{r_1} (|v_{2,n}(\cdot - x_n)|^{r_2} - Q_2^{r_2}) dx\\
&\quad +\,\int_{\R^N} |t_n^\alpha u_{1,n}(t_n (\cdot - x_n)|^{r_1} Q_2^{r_2} dx,
\end{split}
\end{equation}
where
\[
\alpha = \frac{N(2 - r_2)}{2r_1} >0.
\]
We are going to estimate the first term in the right hand side of \eqref{eq.inf.6}. In view of H\"oder's inequality,
\begin{equation}
\label{eq.inf.5}
\begin{split}
&\left|\int_{\R^N} |t_n^\alpha u_{1,n}(t_n (\cdot - x_n))|^{r_1} (|v_{2,n}(\cdot - x_n)|^{r_2} - |Q_2|^{r_2}) dy\right|\\
&\leq\||t_n^\alpha u_{1,n}(t_n(\cdot- x_n))|^{r_1}\|_{\frac{2}{2 - r_2}} \||v_{2,n}(\cdot - x_n)|^{r_2} -  Q_2^{r_2}\|_{\frac{2}{r_2}}\\
&=\|u_{1,n}\|_{\frac{2r_1}{2 - r_2}}^{r_1}\||v_{2,n}(\cdot - x_n)|^{r_2} -  Q_2^{r_2}\|_{\frac{2}{r_2}}.
\end{split}
\end{equation}
From \eqref{eq.inf.3}, it follows that
\[
\left|\int_{\R^N} |t_n^\alpha u_{1,n}(t_n (\cdot - x_n))|^{r_1} (|v_{2,n}(\cdot - x_n)|^{r_2} - |Q_2|^{r_2}) dy\right| = o_n (1).
\]
From \eqref{A1}, there exists \(p > 1\) such that
\[
\frac{2}{2 - r_1} < p\leq\frac{2^*}{r_2}.
\]
By the H\"older inequality,
\[
\begin{split}
\int_{\R^N} |t_n^\alpha u_{1,n}(t_n y)|^{r_1}Q_2(y)^{r_2}  dy &\leq \||t_n^\alpha u_{1,n}(t_n y)|^{r_1}\|_{p}\|Q_2^{r_2}\|_{p'}  \\
&= t_n^{\alpha r_1 p - N}\|u_{1,n}\|_{pr_1}^{r_1}\|Q_2^{r_2}\|_{p'} 
\end{split}
\]
Since \(\alpha r_2 p - N > 0\), the sequence above converges to zero. Consequently, we conclude that
\[
\int_{\R^N} |u_{1,n}|^{r_1} |u_{2,n}|^{r_2} \,dx=o_n(1).
\]
Using Lemma \ref{smin} and \eqref{eq.single}, we obtain
\[
\begin{split}
J(u_{1,n},u_{2,n}) &= I_{\mu_1}(u_{1,n}) + I_{\mu_2}(u_{2,n}) - \beta\int_{\R^N}|u_{1,n}|^{r_1} |u_{2,n}|^{r_2} dx \\
&\geq - \beta\int_{\R^N}|u_{1,n}|^{r_1} |u_{2,n}|^{r_2}dx = o_n(1)
\end{split}
\]
This contradicts the fact that \(m(a_1,a_2^*) < 0\).
Thus we have that $\{u_{2,n}\} \subset S_1 (a_2^*)$ is bounded in $H^1(\R^N)$. Therefore, we derive that any minimizing sequence is bounded in $H^1(\R^N) \times H^1(\R^N)$.

\noindent (iii). 
Let us first consider the case $a_1>a_1^*$ and $a_2>0$. Since $m_{\mu_1}(a_1)=-\infty$ for any $a_1>a_1^*$, see Lemma \ref{smin},
there exists a sequence $\{u_{1,n}\} \subset S_1 (a_1)$ such that $I_{\mu_1}(u_{1,n}) \to -\infty$ as $n \to \infty$. Notice that, for any $u_2 \in S_1 (a_2)$,
$$
m(a_1, a_2) \leq J(u_{1,n},u_2) \leq I_{\mu_1}(u_{1,n})+I_{\mu_2}(u_2).
$$
Thus we obtain that $m(a_1, a_2)=-	\infty$. 
Similarly, we can show that $m(a_1, a_2)=-\infty$ for any $a_1>0$ and $a_2>a_2^*$. Let us now treat the case $a_1=a_1^*$ and $a_2=a_2^*$. In this case, we take \(u = (Q_1,Q_2)\)
in \eqref{eq.conformal}. Then we see that
\begin{equation*}
\begin{split}
J(u_{1,t}, u_{2,t})&=-\beta t^{\frac N 2(r_1+r_2-2)} \int_{\R^N} |u_1|^{r_1}|u_2|^{r_2} \, dx\\
&=-\beta \mu_1^{-\frac{Nr_1}{4}} \mu_2^{-\frac{Nr_2}{4}}t^{\frac N 2(r_1+r_2-2)} \int_{\R^N} |Q|^{r_1+r_2}\, dx,
\end{split}
\end{equation*}
from which we see that $J(u_{1,t}, u_{2,t}) \to -\infty$ as $t \to \infty$. This leads to $m(a_1^*, a_2^*)=-\infty$.

\noindent (iv).
Define
$$
u_{1,n}:=a_{1,n}^{\frac 12} \frac{Q_t}{\|Q\|_2} \in S_1 (a_{1,n}), \quad u_{2,n}:=a_{2,n}^{\frac 12} \frac{Q_t}{\|Q\|_2} \in S_1 (a_{2,n}).
$$
Then we find that
\begin{equation*}
\begin{split}
m(a_{1,n}, a_{2,n}) \leq J(u_{1,n}, u_{2,n}) &= \frac {a_{1,n} t^2}{2\|Q\|_2^2}\left(1-\frac{a_{1,n}}{a_1^*}\right)\int_{\R^N} |\nabla Q|^2 \,dx +\frac {a_{2,n} t^2}{2\|Q\|_2^2}\left(1-\frac{a_{2,n}}{a_2^*}\right)\int_{\R^N} |\nabla Q|^2 \,dx \\
&\quad  -\frac{\beta t^{\frac N 2 (r_1 + r_2 - 2)} a_{1,n}^{\frac{r_1}{2}}a_{2,n}^{\frac{r_2}{2}}}{\|Q\|_2^{r_1+r_2}}\int_{\R^N} |Q|^{r_1+r_2} \,dx.
\end{split}
\end{equation*}
If the sequence \(m(a_{1,n},a_{2,n})\) is bounded from below, up to extract a sequence
we can suppose that it has a limit. Therefore,
$$
\lim_{n \to \infty} m(a_{1,n},a_{2,n}) \leq -\frac{\beta t^{\frac N 2 (r_1 + r_2 - 2)} (a_1^*)^{\frac{r_1}{2}}(a_2^*)^{\frac{r_2}{2}}}{\|Q\|_2^{r_1+r_2}}\int_{\R^N} |Q|^{r_1+r_2} \,dx.
$$
However, taking \(t > 0\) large enough, we obtain a contradiction. This in turn shows that $m(a_{1,n},a_{2, n}) \to -\infty$ as $n \to \infty$.

\noindent (v). Since $C_0^{\infty}(\R^N)$ is dense in $H^1(\R^N)$, there exist $\{(v_{1,n}, v_{2,n})\} \subset S(b_1, b_2) \cap (C_0^{\infty} (\R^N)\times C_0^{\infty}(\R^N))$ and $\{(w_{1,n}, w_{2,n})\} \subset S(c_1, c_2) \cap (C_0^{\infty}(\R^N) \times C_0^{\infty}(\R^N))$ such that $v_{1,n}, v_{2,n} \geq 0$, $w_{1,n}, w_{2,n} \geq 0$ and
\[
J(v_{1,n}, v_{2,n})\leq m(b_1, b_2) + \frac {1}{2n}, \quad J(w_{1,n}, w_{2,n})\leq m(c_1, c_2)+ \frac {1}{2n}.
\]
Without loss of generality, we may assume that $\text{supp} \, v_{i,n} \cap \text{supp}\, w_{i,n}=\emptyset$ for \(i=1,2\), because the energy functional $J$ is invariant under any translation in $\R^N$. Define
$$
u_{1,n}:=v_{1,n}+w_{1,n} \in S_1(a_1), \quad u_{2,n}:=v_{2,n} +w_{2,n} \in S_1(a_2).
$$
Therefore, we conclude that
\begin{equation*}
\begin{split}
m(a_1, a_2) &\leq J(u_{1,n}, u_{2,n}) = J(v_{1,n}, v_{2,n}) +J(w_{1,n}, w_{2,n})
\\ 
&\leq m(b_1, b_2) +m(c_1, c_2) + \frac 1 n.
\end{split}
\end{equation*}
Thus we have that $m(a_1, a_2) \leq m(b_1, b_2) +m(c_1, c_2)$, because $n \in \N^+$ is arbitrary.
\end{proof}
\begin{proof}[Proof of Theorem~\ref{concentration-compactness}] 
Let $\{(u_{1,n}, u_{2,n})\} \subset S(a_1, a_2)$ be a minimizing sequence to \eqref{min}. Our aim is to prove that $\{(u_{1,n}, u_{2,n})\}$ is compact in $H^1(\R^N) \times H^1(\R^N)$ up to translations. In view of Lemma \ref{inf}, we first have that $\{(u_{1,n}, u_{2,n})\}$ is bounded in $H^1(\R^N) \times H^1(\R^N)$. It then follows from \cite[Lemma I.1]{Lio84b} that both sequences of $\{(u_{1,n}, u_{2,n})\}$ do not vanish. Otherwise, using the H\"older inequality we can prove that the coupling term converges to zero. 
Therefore, we obtain \(m(a_1,a_2)\geq 0\) from \eqref{gn}, 
contradicting (i) or (ii) of Lemma~\ref{inf}. 
Consequently, there exist a sequence $\{y_n\} \subset \R^N$ and $(u_1, u_2) \in H^1(\R^N) \times H^1(\R^N)$ with $ u_i\neq 0 $ such that 
\[
(u_{1,n}(\cdot-y_n), u_{2,n}(\cdot-y_n)) \wto (u_1, u_2)
\]
in $H^1(\R^N) \times H^1(\R^N)$ as $n \to \infty$. 
We set $b_1:=\|u_1\|_2^2 > 0$ and $b_2:=\|u_2\|_2^2 > 0$ and define
\[
v_{1,n}:=u_{1,n}(\cdot-y_n)-u_1, \quad v_{2,n}:=u_{2,n}(\cdot-y_n)-u_2.
\]
We claim that $(v_{1,n}, v_{2,n}) \to (0, 0)$ in $L^p(\R^N) \times L^p(\R^N)$ as $n \to \infty$ for any \(2 < p < 2^*\). 
On the contrary, by \cite[Lemma~I.1]{Lio84b} there exist a 
sequence \(\{z_n\} \subset \R^N\) and \((v_1,v_2) \in H^1(\R^N) \times H^1(\R^N)\)
with $(v_1, v_2) \neq (0, 0)$ such that 
\[
(v_{1,n}(\cdot-z_n), u_{2,n}(\cdot-z_n)) \wto (v_1, v_2)
\]
in $H^1(\R^N) \times H^1(\R^N)$ as $n \to \infty$. We set \(c_1 := \|v_1\|_2^2 \geq 0\) and \(c_2 := \|v_2\|_2^2 \geq 0\) and define
\[
w_{1,n} := v_{1,n}(\cdot - z_n) - v_1, \quad w_{2,n} := v_{2,n}(\cdot - z_n) - v_2.
\] 
Applying two times \cite[Lemma 3.2]{GJ16}, we obtain that
\begin{align} \label{sub0}
\begin{split}
m(a_1, a_2)&=J(u_{1,n}, u_{2,n}) +o_n(1)= J(v_{1,n},v_{2,n}) + J(u_1,u_2) + o_n (1)
\\
&=J(w_{1,n}, w_{2,n}) +J(v_1,v_2) +J(u_1, u_2) +o_n(1) \\
& \geq m(a_1-b_1-c_1, a_2-b_2-c_2) + m(c_1, c_2) +m(b_1, b_2) +o_n(1).
\end{split}
\end{align}
On the other hand, using the assertion $(\textnormal{v})$ of Lemma \ref{inf}, we have that
\begin{align} \label{sub1}
\begin{split}
m(a_1, a_2) &\leq m(a_1-b_1-c_1, a_2-b_2-c_2) +m(b_1+c_1, b_2+c_2) \\
&\leq m(a_1-b_1-c_1, a_2-b_2-c_2) + m(c_1, c_2) +m(b_1, b_2).
\end{split}
\end{align}
Therefore, from \eqref{sub0} and \eqref{sub1}, we know that
\begin{align} \label{sub2}
m(b_1+c_1, b_2+c_2) =m(c_1, c_2) +m(b_1, b_2)
\end{align}
and
\begin{align} \label{sub3}
m(c_1, c_2)=J(v_1, v_2), \quad m(b_1, b_2)=J(u_1, u_2).
\end{align}
Observe that 
$$
m(c_1, c_2)=J(|v_1|^*, |v_2|^*), \quad m(b_1, b_2)=J(|u_1|^*,|u_2|^*),
$$
where $f^*$ stands for the symmetric-decreasing rearrangement of $f$. Without restriction, we shall assume that $u_1, u_2, v_1$ and $v_2$ are nonnegative and
radially symmetric. Otherwise, we can replace them with their symmetric-decreasing rearrangements. 
For any $n \in \N^+$, we now define a function $\chi_n: \R^N \to [0, 1]$ by
$$
\chi_n (x): = 
\begin{cases}
1 & \text{ if } |x|\leq n,\\
n + 1 - |x| & \text{ if } n < |x|\leq n + 1,\\
0 & \text{ if } |x| > n + 1.
\end{cases}
$$
Clearly, we can compute that
$$
\nabla \chi_n (x) = 
\begin{cases}
0& \text{ if } |x|\leq n,\\
- \frac{x}{|x|} & \text{ if } n < |x|\leq n + 1,\\
0 & \text{ if } |x| > n + 1,
\end{cases}
$$
and $|\nabla \chi_n(x)| \leq 1$ for any $x \in \R^N$. Define
\begin{align*}
g_{i,n} (x) := \chi_n (x) u_i (x), \quad x \in \R^N,
\end{align*}
and
\begin{align*}
 h_{i,n} (x) &:= \chi_n (x - (n + 1) e_N) v_i (x - (n + 1) e_N), \quad x \in \R^N,
\end{align*}
where $e_N=(0,\cdots, 0, 1) \in \R^N$ and $i=1, 2$. Note that $\text{supp}\, g_{i,n} \cap \text{supp}\, h_{i,n}=\emptyset$ for $i=1,2$. Then
$$
\|f_{i,n}\|_2^2:=\|g_{i,n}\|_2^2+\|h_{i,n}\|^2 = b_i+c_i+o_n(1).
$$
Furthermore, from \cite[Lemma 3.1]{Gar12}, we get that
\begin{align} \label{lem.coupled-single.3}
\|\nabla f_{i, n}^{*}\|_2^2\leq \|\nabla f_{i,n}\|_2^2 - \frac{3}{4}\left(\|\nabla g_{i,n}\|^2 _{L^2(\R\times U_{n,i})} + 
\|\nabla h_{i, n}\|^2_{L^2(\R\times V_{n,i})}\right),
\end{align}
where $U_{i, n}$ and $V_{i, n}$ are subsets of \(\R^{N - 1}\) defined by
\begin{equation*}
\begin{split}
U_{i,n} &:= \big\{x'\in\R\sp{N - 1} \mid \sup_{\R} g_{i,n} (x',\cdot) \leq \sup_{\R} h_{n,i} (x',\cdot)\big\}\\
V_{i,n} &:= \big\{x'\in\R\sp{N - 1} \mid \sup_{\R} h_{i,n} (x',\cdot) < \sup_{\R} g_{n,i} (x',\cdot)\big\}.
\end{split}
\end{equation*}
We also define
\begin{equation*}
\begin{split}
U_i &:= \big\{x'\in\R\sp{N - 1} \mid| \sup_{\R} u_{i} (x',\cdot) \leq \sup_{\R} v_{i} (x',\cdot)\big\}\\
V_i &:= \big\{x'\in\R\sp{N - 1} \mid| \sup_{\R} v_{i} (x',\cdot) < \sup_{\R} u_{i} (x',\cdot)\big\}.
\end{split}
\end{equation*}
It is simple to see that $U_i \subset U_{i,n}$ for $i=1,2$. Indeed, note first that
\begin{equation*}
\begin{split}
\sup_{x_N\in\R} g_{i,n} (x',x_N) &= g_{i,n} (x',0) = \chi_n (x',0)u_i(x',0)\\
\sup_{x_N \in\R} h_{i,n} (x',x_N) &= h_{i,n} (x', n + 1) = \chi_n (x',0)v_i(x',0).
\end{split}
\end{equation*}
If $x' \in U_i$, then $u_i(x',0)\leq v_i(x',0)$. It then follows that 
\[
\chi_n (x',0)u_i(x',0)\leq \chi_n (x',0)v_i(x',0).
\]
Therefore, we have that $x' \in U_{i,n}$. Similarly, we can show that 
\[
V_i\cap B_{n + 1}^{N - 1}(0) \subset V_{i,n}\cap B_{n + 1}^{N - 1}(0)
\]
where \(B_{n + 1}^{N - 1}\) is the ball in 
\(\mathbb{R}^{N - 1}\) with center at the origin and radius \(n + 1\). 
Thus we are able to obtain that
\begin{equation*}
\begin{split}
\|\nabla g_{i,n}\|^2 _{L^2(\R\times U_{n,i})}&\geq  \|\nabla g_{i,n}\|^2 _{L^2(\R\times U_i)} = \|\nabla (\chi_n u_i)\|^2 _{L^2(\R\times U_i)} \\
&= \|u_i\nabla\chi_n + \chi_n\nabla u_i\|^2 _{L^2(\R\times U_i)}\\
&\geq \|\chi_n\nabla u_i\|_{L^2(\R\times U_i)}^2 
- 2|\text{Re}(u_i \nabla \chi_n, \chi_n \nabla u_i)_{L^2(\R\times U_i)}|\\
&\geq \|\nabla u_i\|_{L^2((\R\times U_i)\cap B_n (0))}^2  - 2\|u_i\chi_n\|_{L^2 ((\R\times U_i)\cap C_n)}\|\nabla\chi_n\cdot\nabla u_i\|_{L^2((\R\times U_i)\cap C_n)} \\
&=\|\nabla u_i\|_{L^2(\R\times U_i)}^2 + o_n(1)
\end{split}
\end{equation*}
where \(C_n\) is the subset of \(\mathbb{R}^N\) defined as \(B(0,n+1)\cap\overline{B(0,n)}^c\). Moreover,
\begin{equation*}
\begin{split}
\|\nabla h_{i,n}\|^2 _{L^2(\R\times V_{i,n})} 
&\geq \|\nabla h_{i,n}\|^2 _{L^2((\R\times V_{i,n})\cap B_{n + 1}(0,n+1))}\\
&\geq \|\nabla h_{i,n}\|^2 _{L^2((\R\times V_i)\cap B_{n + 1}(0,n+1))} \\
&=\|v_i\nabla\chi_n + \chi_n \nabla v_i\|^2 _{L^2((\R\times V_i)\cap B_{n + 1}(0))}\\
&\geq \|\chi_n \nabla v_i\|^2 _{L^2((\R\times V_i)\cap B_{n + 1}(0))} - 2\left|\text{Re}\left(v_i \nabla \chi_n, \chi_n \nabla v_i\right)_{L^2((\R\times V_i)\cap B_{n + 1}(0))} \right| \\
&\geq \|\nabla v_i\|^2 _{L^2((\R\times V_i)\cap B_{n + 1}(0))} - 2\,\|v_i \chi_n\|_{L^2((\R\times V_i)\cap C_n)} \| \nabla \chi_n \cdot \nabla v_i\|_{L^2((\R\times V_i)\cap C_n)} \\
&=\|\nabla v_i\|^2 _{L^2(\R\times V_i)} + o_n(1).
\end{split}
\end{equation*}
Taking the limit in \eqref{lem.coupled-single.3} as $n \to \infty$, we then obtain that
$$
\lim_{n\to +\infty} \|\nabla f_{i,n}^{*}\|_{L^2}^2\leq \lim_{n\to\infty}
\|\nabla f_{i,n}\|_{L^2}^2 - \frac{3}{4}\left( \|\nabla u_i\|_{L^2(\R\times U_i)}^2+\|\nabla v_i\|^2 _{L^2(\R\times V_i)}\right).
$$
We set \(b := (b_1,b_2)\) and \(c = (c_1,c_2)\). 
Therefore, we conclude that
\begin{align*}
\begin{split}
m(b) + m(c) &= J(u) + J(v)
=J(g_{n})+J(h_{n}) +o_n(1) 
\geq J(f_{n}) + o_n(1) \\
& \geq  J(f_{n}^*) 
+\frac{3}{4}\sum_{i = 1}^2 \left(\|\nabla v_i\|^2 _{L^2(\R\times V_i)} + \|\nabla u_i\|_{L^2(\R\times U_i)}^2\right) +o_n(1)\\
& \geq  m(b+c) + \frac{3}{4}\sum_{i = 1}^2 \left(\|\nabla v_i\|^2 _{L^2(\R\times V_i)} + \|\nabla u_i\|_{L^2(\R\times U_i)}^2\right) +o_n(1),
\end{split}
\end{align*}
where \(f_n^* = (f_{1,n}^*,f_{2,n}^*)\). Since \(v_1\) and \(u_1\)
are radially symmetric-decreasing, the quantity
\[
\|\nabla u_1\|_{L^2 (\R\times U_1)}^2 + \|\nabla v_1\|_{L^2 (\R\times V_1)}^2
\] 
is positive. Otherwise, if \(U_1\neq\emptyset\), \(u_1\) would be constant
on the unbounded set \(\R\times U_1\). This implies that \(u_1\equiv 0\),
contradicting \(b_1 > 0\), while, if \(U_1 = \emptyset\), then 
\(\|\nabla v_1\|_2^2 = 0\), contradicting \(b_2 > 0\).
Therefore,
\[
m(b_1+c_1, b_2+c_2) <m(b_1, b_2) + m(c_1, c_2).
\]
We then reach a contradiction with \eqref{sub2}. 
This readily infers that $(v_{1,n}, v_{2,n}) \to (0, 0)$ in $L^p(\R^N) \times L^p(\R^N)$ as $n \to \infty$ for any \(2 < p < 2^*\). 
\[
\begin{split}
m(a_1,a_2) &= \frac{1}{2}\lim_{n\to\infty} (\|\nabla v_{1,n}\|_2^2 + \|\nabla v_{2,n}\|_2^2) + m(b_1,b_2)\\
&\geq \frac{1}{2}\lim_{n\to\infty} (\|\nabla v_{1,n}\|_2^2 + \|\nabla v_{2,n}\|_2^2) + m(a_1,a_2).
\end{split}
\]
The last inequality follows from (i), (ii) and (v) of Lemma~\ref{inf}. 
All the inequalities are equalities. Therefore, \(m(a) = m(b)\) and
\(\{\nabla v_{1,n}\}\) and \(\{\nabla v_{2,n}\}\) converge to zero as well. 

To achieve the desired result, it suffices to show that $b_1=a_1$ and $b_2=a_2$. For this, we only need to demonstrate that if $(b_1, b_2) \neq (a_1, a_2)$, then $m(a_1, a_2)<m(b_1, b_2)$. If $b_1<a_1$ and $b_2<a_2$, by Lemma \ref{inf}, then $m(a_1-b_1, a_2-b_2)<0$. Then $m(a_1, a_2)<m(b_1, b_2)$ by the assertion (v) of Lemma \ref{inf}. We now consider the case $b_1=a_1$ and $b_2<a_2$. We assume by contradiction that $m(a_1, a_2)=m(a_1, b_2)$. Define a function \(\chi\colon [0,+\infty)\to\R\) by \(\chi(t):= J(u_1,tu_2)\). Thus
\[
\begin{split}
\chi(t) &= I_{\mu_1} (u_1) + \frac{t^2}{2}\int_{\R^N}|\nabla u_2|^2 \, dx - 
\frac{\mu_2 Nt^{2 + \frac{4}{N}}}{2N + 4}\int_{\R^N} |u_2|^{2 + \frac{4}{N}} \,dx  - \beta t^{r_2}
\int_{\R^N} |u_1|^{r_1} |u_2|^{r_2}dx.
\end{split}
\]
We are going to verify that \(\chi'(1) = 0\). In fact, if \(\chi'(1) > 0\), then there exist \(0< \varepsilon <1\) such that
\(J(u_1,(1 - \varepsilon)u_2) < J(u_1, u_2)\). This then suggests that  
\(m(a_1,(1 - \varepsilon)^2 b_2) < m(a_1,b_2)\). However, this is impossible, because of
\[
m(a_1,b_2)\leq m(a_1, (1 - \varepsilon)^2b_2) + 
m(0, b_2-(1 - \varepsilon)^2b_2) = m(a_1,(1 - \varepsilon)^2b_2).
\]
where the first inequality follows from the assertion (v) of Lemma \ref{inf} and the 
second one follows from Lemma~\ref{smin}. If \(\chi'(1) < 0\), then there exist
\(0 < \varepsilon < 1\) satisfying \((1 + \varepsilon)^2 b_2 < a_2\) such that
\(J(u_1,(1 + \varepsilon)u_2) < J(u_1, u_2)\). 
This then indicates that \(m(a_1,a_2)\leq m(a_1,(1 + \varepsilon)^2b_2) < m(a_1,b_2)\), which contradicts our assumption that $m(a_1, a_2)=m(a_1, b_2)$.
As a consequence, \(\chi'(1) = 0\). Then
\[
\chi'(1) = \int_{\R^N} |\nabla u_2|^2\, dx - \mu_2 \int_{\R^N}|u_2|^{2 + \frac4N} \, dx - 
\beta r_2 \int_{\R^N} |u_1|^{r_1} |u_2|^{r_2}\,dx = 0.
\]
Multiplying the second equation in \eqref{elliptic-sys} by $u_2$ and integrating on $\R^N$, we then have that
 \(\lambda_2 = 0\). However, by applying the Liouville type result (see \cite[Lemma~2.3]{Gou18}) to the equation in \eqref{eq.thm.1} satisifed by $|u_2|$, we know that $\lambda_2>0$,
as long as \(N\leq 4\). This is a contradiction. Therefore, we have that $m(a_1, a_2)<m(a_1, b_2)$. Similarly, we can prove that $m(a_1, a_2)<m(b_1, a_2)$ if
\(b_1 < a_1\). This completes the proof.
\end{proof}
\subsection{Characterization of minimizers} In this subsection, we shall reveal some properties of minimizers to \eqref{min} and give the proof of Theorem \ref{thm}.
\begin{proof}[Proof of Theorem \ref{thm}] Let $(u_1, u_2) \in S(a_1, a_2)$ be
such that $J(u_1, u_2)=m(a_1,a_2)$. Note that $(|u_1|, |u_2|) \in S(a_1, a_2)$ and
\begin{equation}
\label{eq.gradient-1}
\int_{\R^N} |\nabla |u_1||^2 \,dx \leq \int_{\R^N} |\nabla u_1|^2 \,dx, \quad \int_{\R^N} |\nabla |u_2||^2 \,dx \leq \int_{\R^N} |\nabla u_2|^2 \,dx.
\end{equation}
Then we have that $J(|u_1|, |u_2|)=m(a_1, a_2)$. This further gives rise to
\begin{equation}
\label{eq.gradient-2}
\int_{\R^N} |\nabla |u_1||^2 \,dx = \int_{\R^N} |\nabla u_1|^2 \,dx, \quad \int_{\R^N} |\nabla |u_2||^2 \,dx = \int_{\R^N} |\nabla u_2|^2 \,dx.
\end{equation}
For simplicity, we shall write $w_1=|u_1|$ and $w_2=|u_2|$ in the following. Since $J(w_1, w_2)=m(a_1, a_2)$, then $(w_1, w_2)$ satisfies weakly the nonlinear elliptic system
\begin{align} 
\label{eq.thm.1}
\left\{
\begin{aligned}
-\Delta w_1 +\lambda_1 w_1 = \mu_1 w_1^{1 + \frac{4}{N}} + \beta r_1 w_1^{r_1 - 1} w_2^{r_2}, \\
-\Delta w_2 + \lambda_2 w_2 =\mu_2 w_2^{1 + \frac{4}{N}} + \beta r_2 w_1^{r_1} w_2^{r_2 - 1}.
\end{aligned}
\right.
\end{align}
(a). Solutions are classical and do not have zeros. By local regularity results, the system is
satisfied in the classical sense. In fact, we can write the first equation as
\(L w_1 = f_1(x)\), where 
\[
f_1 := \lambda_1 w_1 - \mu_1 w_1^{1 + \frac{4}{N}} - \beta w_1^{r_1 - 1} 
w_2^{r_2}
\]
and \(L\) is the Laplacian operator.  
We show that \(f_1\) is in \(L^p _{loc}\) for some \(1 < p\). Showing that the interaction term
is locally integrable requires some estimates in the case \(N\geq 3\).
We are seeking \(p > 1\) such that 
\(
w_1^{p(r_1-1)} w_2^{pr_2}\text{ is in } L^1 _{loc}.
\)
To apply the H\"older inequality, we need to ensure the existence of \(q > 1\) 
such that 
\[
1\leq qpr_2\leq\frac{2N}{N - 2},\quad 1\leq\frac{q}{q - 1}\cdot p(r_1 - 1)\leq 
\frac{2N}{N - 2}.
\]
We set the restriction \(1 < p(r_1 - 1) < 2^*\). The two inequalities
are equivalent to
\[
\frac{2N}{2N - (N - 2)p(r_1 - 1)}\leq q\leq\frac{2N}{(N - 2)pr_2},\quad q\leq\frac{1}{p(r_1 - 1) - 1}.
\]
For the existence of such \(q > 1\), it is enough to have
\[
\frac{2N}{2N - (N - 2)p(r_1 - 1)}\leq \frac{2N}{(N - 2)pr_2},\quad
\frac{2N}{(N - 2)pr_2} > 1,\quad p(r_1 - 1) < 2^*.
\]
Equivalently,
\[
p < 2^*\min\left\{\frac{1}{r_1 + r_2 - 1},\frac{1}{r_2},
\frac{1}{r_1 - 1}\right\}.
\]
In conclusion, there exist \(p > 1\) such that \(f_1\) is \(L^p _{loc}\), provided
\[
1 < \frac{2^*}{\max\{r_1 + r_2 - 1,r_2,r_1 - 1\}} = 
\frac{2^*}{r_1 + r_2 - 1}.
\]
From \eqref{AH}, the righthand side can be estimated from below by
\[
2 > \frac{2N^2}{(N - 2)(N + 4)} := p_0 > 1
\]
Therefore, the function \(f_1\) is in \(L^{p_0}_{loc}\) as it can be written as
the sum of three functions in \(L^{p_0} _{loc}\), with \(p_0\) depending on
the dimension of the domain space.
We can then apply results on local 
\(L^p\) regularity, when \(p\) is in \((1,2)\),
as in \cite[\S7]{GM12}, and the bootstrapping procedure described in \cite[\S8.4.3]{Eva10} to claim that both \(w_1\) and \(w_2\) are 
two-times continuously differentiable functions. Therefore, from
Hopf's Lemma, \cite[Lemma 1,~p. 519]{Eva10}, we obtain $w_1, w_2>0$. Therefore, \(u_1\) and \(u_2\) do not have zeroes.

(b). Solutions are strictly symmetric-decreasing. Observe that $(w_1^*, w_2^*) \in S(a_1, a_2)$ and
$$
\int_{\R^N} |\nabla w_1^*|^2 \,dx \leq \int_{\R^N} |\nabla w_1|^2 \,dx, \quad \int_{\R^N} |\nabla w_2^*|^2 \,dx \leq \int_{\R^N} |\nabla w_2|^2 \,dx,
$$
$$
\int_{\R^N} |w_1^*|^{2+\frac 4N} \,dx=\int_{\R^N} |w_1|^{2+\frac 4N} \,dx, \quad \int_{\R^N} |w_2^*|^{2+\frac 4N} \,dx=\int_{\R^N} |w_2|^{2 +\frac 4 N} \,dx.
$$
In addition, there holds that
$$
\int_{\R^N} |w_1|^{r_1} |w_2|^{r_2} \,dx \leq \int_{\R^N} |w_1^*|^{r_1} |w_2^*|^{r_2} \,dx. 
$$
As a result, we have that $m(a_1, a_2)=J(w_1^*, w_2^*)=J(w_1, w_2)$. This in turn leads to 
\begin{align} \label{w12}
\int_{\R^N} |\nabla w_1^*|^2 \,dx = \int_{\R^N} |\nabla w_1|^2 \,dx, \quad \int_{\R^N} |\nabla w_2^*|^2 \,dx = \int_{\R^N} |\nabla w_2|^2 \,dx
\end{align}
and
\begin{equation}
\label{w13}
\int_{\R^N} |w_1|^{r_1} |w_2|^{r_2}dx = \int_{\R^N} |w_1^*|^{r_1} |w_2^*|^{r_2}dx. 
\end{equation}
According to \cite[Theorem~1.1]{BZ88}, $w_1$ and $w_2$ are symmetric-decreasing
provided the two sets
\begin{gather*}
C^*_1 := \{x\in\R^N\mid\nabla w_1^* (x) = 0\}\cap {w_1^*}^{-1} (0,\sup w_1^*),\\
C^*_2 := \{x\in\R^N\mid\nabla w_2^* (x) = 0\}\cap {w_2^*}^{-1} (0,\sup w_2^*)
\end{gather*}
have zero measure. In what follows, we are going to show that $C_1^*=\emptyset$ and $C_2^*=\emptyset$. We rely on an argument used in 
\cite[Proof~of~Theorem~1.4]{ZL18}. Since $(w_1^*, w_2^*) \in S(a_1, a_2)$ is a minimizer to \eqref{min}, then $(w_1^*, w_2^*)$ solves the nonlinear elliptic system
\begin{align} \label{eq.minima-structure.5}
\left\{
\begin{aligned}
-\Delta w_1^* +\lambda_1 w_1^* = \mu_1 {w_1^*}^{1 + \frac{4}{N}} + 
\beta r_1 {w_1^*}^{r_1 - 1}{w_2^*}^{r_2}, \\
-\Delta w_2^* + \lambda_2 w_2^* =\mu_2 {w_2^*}^{1 + \frac{4}{N}} + 
\beta r_2 {w_1^*}^{r_1} {w_2^*}^{r_2 - 1}.
\end{aligned}
\right.
\end{align}
Since $w_1^*, w_2^* \in H^1(\R^N)$, by using elliptic regularity theory just as in (a), we have $w_1^*, w_2^* \in C^2(\R^N)$. In addition, we have that $w_1^*(x), w_2^*(x) \to 0$ as 
$|x| \to \infty$, because \(w_i^*\) are positive, symmetric-decreasing and \(L^2\). We prove that $C_1^*=\emptyset$. 
On the contrary, there exists a point  $z_0 \in \R^N$ such that $z_0 \in C_1^*$. That is,
\begin{align} \label{assume}
\nabla w_1^*(z_0)=0, \quad w_1^*(z_0) < \sup w_1^*=w_1^*(0).
\end{align}
Define $\Omega:=\{x \in \R^N : w_1^*(x) \geq w_1^*(z_0)\}$. Since $w_1$ enjoys the first equation in \eqref{eq.minima-structure.5}, then
\begin{equation} 
\label{equw}
\begin{split}
-\Delta (w_1^* - w_1^*(z_0)) &+ \lambda_1 (w_1^* - w_1^*(z_0)) = \mu_1 
{w_1^*}^{1 + \frac{4}{N}}\\ 
&+ \beta r_1 {w_1^*}^{r_1 - 1}{w_2^*}^{r_2} - 
\lambda_1 w_1^* (z_0)\text{ on}\ \Omega.
\end{split}
\end{equation}
Since \(w_1^*\) is radially symmetric-decreasing and $\nabla w_1^*(z_0)=0$, then $\Delta w_1^* (z_0) = 0$. Therefore, from \eqref{equw}, we obtain that
$$
\lambda w_1^* (z_0) =  \mu_1 {w_1^*(z_0)}^{1 + \frac{4}{N}} +  \beta r_1 {w_1^*(z_0)}^{r_1 - 1}{w_2^*(z_0)}^{r_2}.
$$
It then follows that
\begin{equation*}
\begin{split}
-\Delta (w_1^* - w_1^*(z_0)) &+ \lambda_1 (w_1^* - w_1^*(z_0)) = \mu_1
{w_1^*}^{1 + \frac{4}{N}} - {w_1^*(z_0)}^{1 + \frac{4}{N}} \\
&+\beta r_1 {w_1^*}^{r_1 - 1} {w_2^*}^{r_1}- {w_1^*(z_0)}^{r_1 - 1}{w_2^*(z_0)}^{r_2}\ \text{on}\ \Omega.
\end{split}
\end{equation*}
By the definition of $\Omega$ and the fact that \(w_1^*\) and \(w_2^*\)
are symmetric decreasing, we have 
$$
-\Delta (w_1^* - w_1^*(z_0)) + \lambda_1 (w_1^* - w_1^*(z_0)) \geq 0 \quad \text{in} \,\, \Omega.
$$
Again, from \cite[Lemma 1,\ p.~519]{Eva10} to conclude that $w_1^*=w_1^*(z_0)$ on $\Omega$, because of $\nabla w_1^*(z_0)=0$.
Since $0 \in \Omega$, then $w_1^*(z_0)=w_1^*(0)$. This contradicts the assumption made in \eqref{assume} that $w_1^*(z_0)<w_1^*(0)$. Hence we get that $C_1^*=\emptyset$. A similar argument applies to \(C_2^*\). Therefore, $w_1$ and $w_2$ are strictly symmetric-decreasing with respect to points \(x_1\) and \(x_2\), respectively, and vanishing
at infinity. From \eqref{w13},
$$
\int_{\R^N} {(w_1 ^{r_1})}^* {(w_2 ^{r_2})}^* \,dx = \int_{\R^N} w_1^{r_1} w_2 ^{r_2} \,dx = \int_{\R^N} {w_1 ^{r_1}}(\cdot - x_1) {w_2 ^{r_2}} (\cdot - x_1) \,dx.
$$
Making use of \cite[Theorem 3.4]{LL01}, we obtain $ w_2(\cdot-x_1) = w_2^* $. 
This implies that $x_1=x_2$. We set \(x_0 := x_1 = x_2\). From the equalities in \eqref{eq.gradient-1} and \eqref{eq.gradient-2}, we can apply \cite[Theorem 4.1]{HS04}. Then,
there are $\theta_1, \theta_2 \in \mathbb{R}$ such that $u_1=e^{\textnormal{i} \theta_1} |u_1|$ and $u_2=e^{\textnormal{i} \theta_2} |u_2|$. Setting \(R_i := w_i (\cdot - x_0)\),
we obtain 
\[
(u_1(x),u_2(x)) = (e^{\textnormal{i} \theta_1} R_1 (x - x_0),e^{\textnormal{i} \theta_2} R_2 (x - x_0)).
\]
\end{proof}

\subsection{Orbital stability of minimizers} We are now going to discuss orbital stability of minimizers to \eqref{min}, whose proof is based on the methods displayed in \cite{CL82}.
\begin{proof}[Proof of Theorem \ref{stability}] The proof of orbital stability of the set $G:=G(a_1, a_2)$ is completed by contradiction arguments. We assume that there exists a sequence 
\[
\{(\phi_{1,0}^n, \phi_{2,0}^n)\}\subset H^1(\R^N) \times H^1(\R^N)
\]
satisfying $d((\phi_{1,0}^n, \phi_{2,0}^n), G)=o_n(1)$, a constant $\eps_0>0$ and a sequence $\{t_n\} \subset \R^+$ such that
\begin{align} \label{unstable}
d((\phi_{1,n}(t_n, \cdot), \phi_{2,n}(t_n, \cdot)), G) \geq \eps_0,
\end{align}
where $(\phi_{1,n}, \phi_{2,n}) \in C_t H^1_x([0, \infty) \times \R^N) \times C_t H^1_x([0, \infty) \times \R^N)$ is the solution to the Cauchy problem for \eqref{sys} with the initial datum $(\phi_{1,0}^n, \phi_{2,0}^n)$. The 
conservation and the continuity of energy and masses give
\begin{gather*}
J(\phi_{1,n}(t_n, \cdot), \phi_{2,n}(t_n, \cdot)) = 
J(\phi_{1,0}^n, \phi_{2,0}^n) = m(a_1,a_2) + o_n (1)\\
\label{eq.orb.2}
\|\phi_{i,n}(t_n, \cdot)\|_2^2 = 
\|\phi_{0,i}(t_n, \cdot)\|_2^2 = a_i + o_n (1)
\end{gather*}
for \(i=1,2\). Therefore we see that
\begin{equation}
\label{eq.}
(u_{1,n},u_{2,n}) := \left(\frac{\sqrt{a_1}\phi_{1,n}(t_n,\cdot)}%
{\|\phi_{1,0}^n\|_2},\frac{\sqrt{a_2}\phi_{2,n}(t_n,\cdot)}%
{\|\phi_{2,0}^n\|_2}\right)\in S(a_1,a_2)
\end{equation}
is a minimizing sequence for \(J\).
From Theorem~\ref{concentration-compactness}, there exist \(u\) in \(G\) and
a sequence \(\{y_n\}\) in \(\R^N\) such that \(u_n(\cdot - y_n)\to u\) in 
\(H^1(\R^N)\times H^1 (\R^N)\). Therefore,
\[
\begin{split}
d(\phi_{n}(t_n,\cdot),G) &\leq d(\phi_{n}(t_n,\cdot),u_n) + d(u_n,G) \\
&= d(\phi_{n}(t_n,\cdot),u_n) + d(u_n(\cdot - y_n),G) = o_n(1).
\end{split}
\]
The first distance converges to zero from \eqref{eq.orb.2}, while the convergence
of the second distance follows from the Concentration-Compactness behaviour
of \(u_n\). Therefore, we reach a contradiction with \eqref{unstable}. 
Hence the set $G(a_1, a_2)$ is orbitally stable.
\end{proof}
\subsection{Concentration of minimizers} We are now in a position to analyze concentration of minimizers to \eqref{min} and present the proof of Theorem \ref{concentration-blowup}.

\begin{proof}[Proof of Theorem \ref{concentration-blowup}] Let $\{(u_{1,n}, u_{2,n})\} \subset S(a_1, a_2)$ be a minimizer to \eqref{min}. 
Then $(u_{1,n},u_{2,n})$ solves the nonlinear elliptic system
\begin{align} \label{equu}
\left\{
\begin{aligned}
-\Delta u_{1,n} + \lambda_{1,n} u_{1,n} &= \mu_1 |u_{1,n}|^{\frac 4N}u_{1,n} + \beta r_1 |u_{1,n}|^{r_1-2} u_{1,n} |u_{2,n}|^{r_2},\\
-\Delta u_{2,n} + \lambda_{2,n} u_{2,n} &= \mu_2 |u_{2,n}|^{\frac 4N}u_{2,n} + \beta r_2 |u_{2,n}|^{r_1-2}u_{1,n} |u_{1,n}|^{r_2}.
\end{aligned}
\right.
\end{align}
From the Pohozaev identity, it follows that
\begin{align} \label{ph}
\begin{split}
\frac{1}{2} \int_{\R^N} |\nabla u_{1,n}|^2+|\nabla u_{2,n}|^2 dx &=\frac{N}{2N+4} \int_{\R^N} \mu_1|u_{1,n}|^{2+\frac 4 N} + \mu_2 |u_{2,n}|^{2 +\frac 4N} \, dx \\
&\quad +\frac{N(r_1 + r_2 -2)}{4}\beta \int_{\R^N} |u_{1,n}|^{r_1}|u_{2,n}|^{r_2} \, dx.
\end{split}
\end{align}
From \eqref{AH}, \eqref{gn} and \eqref{ph} it follows that
\begin{equation*}
\begin{split}
J(u_{1,n},u_{2,n})&=\left(\frac{N(r_1 + r_2 -2)}{4}-1\right) \beta\int_{\R^N}|u_{1,n}|^{r_1}|u_{2,n}|^{r_2} \, dx \\
&\geq C\left(\frac{N(r_1 + r_2 -2)}{4}-1\right) \left(\int_{\R^N} |\nabla u_{1,n}|^2+|\nabla u_{2,n}|^2\, dx \right)^{\frac{N(r_1 + r_2 -2)}{4}}.
\end{split}
\end{equation*}
where $C=C(N, r_1, r_2, a_1, a_2, \beta)$. According to (iv) of Lemma~\ref{inf},
\[
\eps_n^{-2}:=\int_{\R^N} |\nabla u_{1,n}|^2+|\nabla u_{2,n}|^2 \,dx \to \infty \quad \mbox{as} \, \, n \to \infty.
\]
Define
\begin{equation}
\label{eq.sysv.4}
(v_{1,n}, v_{2,n}):= (\eps_n^{\frac N 2}u_{1,n}(\eps_n \cdot), \eps_n^{\frac N 2}u_{2,n}(\eps_n \cdot)) \in S(a_1, a_2).
\end{equation}
Using \eqref{ph}, we get that
\begin{equation} \label{eq.sysv.1}
\begin{split}
\frac {1}{2} \int_{\R^N} |\nabla v_{1,n}|^2+|\nabla v_{2,n}|^2 \,dx &=\frac{N}{2N+4} \int_{\R^N} \mu_1|v_{1,n}|^{2+\frac 4 N} + \mu_2 |v_{2,n}|^{2 +\frac 4N} \,dx \\
& +\frac{N(r_1 + r_2 -2)}{4}\beta \eps_n^{2-N\left(\frac{r_1 + r_2}{2} -1\right)}\int_{\R^N} |v_{1,n}|^{r_1}|v_{2,n}|^{r_2}\,dx.
\end{split}
\end{equation}
Since $2<r_1+r_2<2+ \frac 4N$, $\|\nabla v_{1,n}\|_2^2 + \|\nabla v_{2,n}\|^2 _2 = 1$ and $\eps_n=o_n(1)$, then
\begin{align} \label{l1}
\int_{\R^N} \mu_1|v_{1,n}|^{2+\frac 4 N} + \mu_2|v_{2,n}|^{2 +\frac 4N} \,dx =\frac{N+2}{N} +o_n(1).
\end{align}
From \cite[Lemma I.1]{Lio84b}, there exists a sequence $\{y_n\} \in \R^N$ and a nontrivial \((v_1, v_2)\) in \(H^1(\R^N) \times H^1(\R^N)\) such that 
\[
(v_{1,n}(\cdot - y_n), v_{2,n}(\cdot -y_n)) \wto (v_1, v_2)
\]
in $H^1(\R^N)\times H^1(\R^N)$ as $n \to \infty$.
From \eqref{equu}, we obtain
\begin{equation} \label{equv}  
\left\{
\begin{split} 
-\Delta v_{1,n} &+ \eps_n^2\lambda_{1,n} v_{1,n} = \mu_1 |v_{1,n}|^{\frac 4N}v_{1,n} + \beta r_1\eps_n^{2-N\left(
\frac{r_1 + r_2}{2}-1\right)}|v_{1,n}|^{r_1-2} v_{1,n} |v_{2,n}|^{r_2},\\
-\Delta v_{2,n} &+ \eps_n^2\lambda_{2,n} v_{2,n} = 
\mu_2 |v_{2,n}|^{\frac 4N}v_{2,n} + \beta r_2  
\eps_n^{2-N\left(\frac{r_1 + r_2}{2}-1\right)} |v_{2,n}|^{r_2-2} v_{2,n} |v_{1,n}|^{r_1}.
\end{split}
\right.
\end{equation}
By multiplying the first equation of \eqref{equv} by \(v_{1,n}\) and integrating on $\R^N$, we have that
\begin{equation*}
\begin{split}
\eps_n^2\lambda_{1,n} a_{1,n} &= -\int_{\R^N}|\nabla v_{1,n}|^2 \, dx + 
\mu_1 \int_{\R^N}|v_{1,n}|^{2 + \frac4N} \,dx+ \beta r_1 \eps_n^{2-N\left(\frac{r_1 + r_2}{2}-1\right)} \int_{\R^N} |v_{1,n}|^{r_1}|v_{2,n}|^{r_2} \,dx.
\end{split}
\end{equation*}
It then follows that $\{\eps_n^2\lambda_{1,n}\}$ is bounded. A similar argument applies to show that 
$\{\eps_n^2\lambda_{2,n}\}$ is bounded as well. 
In addition, by multiplying the first equation and the second equation of \eqref{equv} by $v_{1,n}$ and $v_{2,n}$ respectively and integrating on $\R^N$, we get that 
$$
a_1^*\eps_n^2 \lambda_{1,n} + a_2^*\eps_n^2 \lambda_{2,n} =\frac 2 N +o_n(1),
$$ 
where we also used \eqref{l1} and the fact that $\|\nabla v_{1,n}\|_2^2 + \|\nabla v_{2,n}\|^2 _2 = 1$. As a consequence, there are constants $\lambda_1, \lambda_2 \in \R$ such that
\begin{equation}
\label{eq.sysv.3}
\eps_n^2\lambda_{1,n} \to \lambda_1,\quad \eps_n^2\lambda_{2,n} \to \lambda_2
\quad a_1^*\lambda_1+a_2^*\lambda_2=\frac 2N.
\end{equation}
Define $a_1:=\|v_1\|_2^2$ and $a_2:=\|v_2\|_2^2$. Then $a_1 \leq a_1^*$ and $a_2 \leq a_2^*$. By the weak convergence, up to extract a subsequence, we have
\[
\|v_{i,n}(\cdot - y_n) - v_i\|_2^2 = \|v_{i,n}\|_2^2 -\|v_i\|_2^2+o_n(1)=
a_i^* - a_i+o_n(1)
\]
for \(i=1,2\). Taking the limit in \eqref{eq.sysv.1}, we find that
$$
\sum_{i = 1}^2 \|\nabla v_{i,n}\|_{2}^{2}=\frac{N}{N+2} \sum_{i = 1}^2 \mu_i \|v_{i,n}\|_{2 + \frac 4N}^{2 + \frac 4N} +o_n(1).
$$
From \cite[Theorem]{BL83} and \eqref{gn}, it then follows that
\begin{equation*}
\begin{split}
&\sum_{i = 1}^2 \|\nabla v_{i,n}(\cdot-y_n) -\nabla  v_i\|_{2}^{2} + 
\sum_{i = 1}^2 \|\nabla v_i\|_{2}^{2}\\
&=\frac{N}{N+2} \sum_{i = 1}^2 \mu_i \|v_{i,n}(\cdot-y_n) - v_i\|_{2 + \frac 4N}^{2 + \frac 4N} +  \frac{N}{N+2} \sum_{i = 1}^2 \mu_i \|v_i\|_{2 + \frac 4N}^{2 + \frac 4N} +o_n(1)\\
&\leq \sum_{i=1}^2\left(\frac{a_i^* - a_i}{a_i^*}\right)^{\frac2N}\|\nabla v_{i,n}(\cdot-y_n) - \nabla v_i\|_2^2 +
\sum_{i=1}^2 \left(\frac{a_i}{a_i^*}\right)^{\frac2N}\|\nabla v_i\|_2^2+o_n(1).
\end{split}
\end{equation*}
Thus we are able to infer that $a_1=a_1^*$ and $a_2=a_2^*$. Furthermore, we can get that 
$$
\|\nabla v_{1,n}(\cdot-y_n)-\nabla v_1\|_2=o_n(1), \quad \|\nabla v_{2,n}(\cdot-y_n)-\nabla v_2\|_2=o_n(1).
$$ 
Consequently, we have that $(v_{1,n}(\cdot - y_n), v_{2,n}(\cdot -y_n)) \to (v_1, v_2)$ in the space 
$H^1(\R^N)\times H^1(\R^N)$ as $n \to \infty$. From \eqref{equv}, $v_1 $
and $v_2$ solve the system 
\begin{align} \label{equs}
-\Delta v_1 + \lambda_1 v_1 = \mu_1 |v_1|^{\frac 4N}v_1, \quad -\Delta v_2 + \lambda_2 v_2 &=\mu_2|v_2|^{\frac 4N}v_2.
\end{align}
Moreover, \(\lambda_i\geq 0\) for \(i=1,2\). The uniqueness of solutions
to the equations above, up to a conformal rescaling and phase change, implies that
there exist
\(\theta_i\in\R\) and \(y_i\in\R^N\) such that 
\[
v_i (x) = e^{i\theta_i}\left(\frac{\lambda_i N}{2\mu}\right)^{\frac{N}{4}} 
Q\left(\left(\frac{\lambda_i N}{2}\right)^{\frac{1}{2}} (x - y_i)\right),\quad
x\in\R^N,
\]
where $Q$ is the unique positive, radially symmetric solution to \eqref{equq} with $p=2+4/N$. 
Combining \eqref{eq.sysv.3}, \eqref{eq.sysv.4}, the equality above and the
convergence of \(\{v_{i,n}(\cdot - y_n)\}\), we conclude the proof.
\end{proof}
\def\MR#1{\href{http://www.ams.org/mathscinet-getitem?mr=#1}{MR#1}}
\def\cprime{$'$} \def\cprime{$'$} \def\cprime{$'$} \def\cprime{$'$}
  \def\cprime{$'$} \def\cprime{$'$} \def\cprime{$'$} \def\cprime{$'$}
  \def\cprime{$'$} \def\polhk#1{\setbox0=\hbox{#1}{\ooalign{\hidewidth
  \lower1.5ex\hbox{`}\hidewidth\crcr\unhbox0}}} \def\cprime{$'$}
  \def\cprime{$'$} \def\cprime{$'$}
\providecommand{\bysame}{\leavevmode\hbox to3em{\hrulefill}\thinspace}
\providecommand{\MR}{\relax\ifhmode\unskip\space\fi MR }
\providecommand{\MRhref}[2]{%
  \href{http://www.ams.org/mathscinet-getitem?mr=#1}{#2}
}
\providecommand{\href}[2]{#2}

\end{document}